\documentclass[fleqn,reqno,10pt,a4paper,final]{amsart}

\usepackage[a4paper,left=35mm,right=35mm,top=30mm,bottom=30mm,marginpar=25mm]{geometry} 
\usepackage{amsmath}
\usepackage{amssymb}
\usepackage{amsthm}
\usepackage{amscd}
\usepackage[ansinew]{inputenc}
\usepackage{cite}
\usepackage{bbm}
\usepackage{color}
\usepackage[english=american]{csquotes}
\usepackage[final]{graphicx}
\usepackage{hyperref}

\graphicspath{{Pictures/}}

\numberwithin{equation}{section}

\newtheoremstyle{thmlemcorr}{10pt}{10pt}{\itshape}{}{\bfseries}{.}{10pt}{{\thmname{#1}\thmnumber{ #2}\thmnote{ (#3)}}}
\newtheoremstyle{thmlemcorr*}{10pt}{10pt}{\itshape}{}{\bfseries}{.}\newline{{\thmname{#1}\thmnumber{ #2}\thmnote{ (#3)}}}
\newtheoremstyle{defi}{10pt}{10pt}{\itshape}{}{\bfseries}{.}{10pt}{{\thmname{#1}\thmnumber{ #2}\thmnote{ (#3)}}}
\newtheoremstyle{remexample}{10pt}{10pt}{}{}{\bfseries}{.}{10pt}{{\thmname{#1}\thmnumber{ #2}\thmnote{ (#3)}}}
\newtheoremstyle{ass}{10pt}{10pt}{}{}{\bfseries}{.}{10pt}{{\thmname{#1}\thmnumber{ A#2}\thmnote{ (#3)}}}

\theoremstyle{thmlemcorr}
\newtheorem{theorem}{Theorem}
\numberwithin{theorem}{section}

\newtheorem{corollary}[theorem]{Corollary}

\theoremstyle{thmlemcorr*}
\newtheorem{theorem*}{Theorem}
\newtheorem{lemma*}[theorem]{Lemma}
\newtheorem{corollary*}[theorem]{Corollary}
\newtheorem{proposition*}[theorem]{Proposition}
\newtheorem{problem*}[theorem]{Problem}
\newtheorem{conjecture*}[theorem]{Conjecture}

\theoremstyle{defi}
\newtheorem{definition}[theorem]{Definition}

\theoremstyle{remexample}
\newtheorem{remark}[theorem]{Remark}

\theoremstyle{ass}
\newtheorem{assumption}{Assumption}

\newcommand{\Hcal}{\mathcal{H}}

\newcommand{\Xcal}{\mathcal{X}}

\newcommand{\Sbb}{\mathbb{S}}

\DeclareMathOperator{\id}{id}

\newcommand{\di}{\mathrm{d}}
\newcommand{\sinc}{\mathrm{sinc}}

\newcommand{\R}{\mathbb{R}}
\newcommand{\C}{\mathbb{C}}

\newcommand{\eps}{\epsilon}

\def\XXint#1#2#3{{\setbox0=\hbox{$#1{#2#3}{\int}$} 
\vcenter{\hbox{$#2#3$}}\kern-.5\wd0}}

\newcommand{\p}{\partial}

\renewcommand{\eps}{\varepsilon}
\renewcommand{\phi}{\varphi}

\title{Uniform convergence in von Neumann's ergodic theorem in the absence of a spectral gap}

\author{Jonathan Ben-Artzi}
\author{Baptiste Morisse}

\address{School of Mathematics, Cardiff University, Cardiff CF24 4AG, Wales, UK}

\email{Ben-ArtziJ@cardiff.ac.uk}
\email{MorisseB@cardiff.ac.uk}

\begin{document}
\maketitle

\begin{abstract}
Von Neumann's original proof of the ergodic theorem is revisited. A uniform convergence rate is established under the assumption that one can control the density of the spectrum of the underlying self-adjoint operator when restricted to suitable subspaces. Explicit rates are obtained when the bound is polynomial, with applications to the linear Schr\"odinger and wave equations. In particular, decay estimates for time-averages of solutions are shown. 

\vspace{4pt}


\vspace{4pt}


\vspace{4pt}

\noindent\textsc{Acknowledgements:} The authors acknowledge support from Fellowship EP/N020154/1 of the Engineering and Physical Sciences Research Council (EPSRC). The authors are grateful to the anonymous referee who read the manuscript carefully and helped to significantly improve it.

\vspace{4pt}

\noindent\textsc{Date:} \today.
\end{abstract}

\hypersetup{
  pdfauthor = {Jonathan Ben-Artzi (University of Cambridge)},
  pdftitle = {},
  pdfsubject = {MSC (2010): xxxxx (primary); xxxxx, xxxxx},
  pdfkeywords = {}
}

\setcounter{tocdepth}{1} 
\tableofcontents

\section{Introduction}\label{sec:intro}
In this note we obtain  uniform rates of convergence in von Neumann's ergodic theorem for continuous time dynamical systems lacking a spectral gap.  It is well-known that a spectral gap leads to the rate $T^{-1}$, however systems that lack a spectral gap are not as well understood. The key new ingredient is an estimate of the \emph{Density of States} (DoS) of the generator near $0$. Such estimates are readily available for differential operators, as is shown in the examples in Section \ref{sec:ex}.

Our setup is as follows. Let $\Hcal$ be a separable Hilbert space and let $U_t:\Hcal\to\Hcal$ be a one-parameter group of unitary transformations. Let $H:D(H)\subset\Hcal\to\Hcal$ be its self-adjoint generator: $U_t=e^{itH}$. We assume  that $H$ has continuous spectrum in a neighborhood of $0$ (and $0$ itself is often an eigenvalue) and show that a  bound on the DoS near $0$  leads to a uniform convergence rate  on a suitable subspace $\Xcal\subset\Hcal$. We apply this to the linear Schr\"odinger and wave equations, to obtain various decay estimates depending on the space $\Xcal$ (which is taken to either be a weighted-$L^2$ space, or the space $L^1\cap L^2$) which consequently lead to global-in-time results, see \eqref{eq:schr-global}, \eqref{eq:schr-global2}, \eqref{eq:wave-decay-global}, \eqref{eq:wave-decay-global2}.

\subsection{Von Neumann's ergodic theorem}
Von Neumann's ergodic theorem \cite{VonNeumann1932a} is a pillar of modern mathematics.
Defining
	\begin{equation*}
	P^T:=\frac{1}{2T}\int_{-T}^TU_t\,\di t,
	\end{equation*}
and
	\begin{equation*}
	P:=\text{orthogonal projection of $\Hcal$ onto }\ker H,
	\end{equation*}
it states
\begin{theorem}[Ergodic theorem \cite{VonNeumann1932a}]\label{thm:von}
For any $f\in\Hcal$, $P^Tf\to Pf$ as $T\to+\infty$.
\end{theorem}

\begin{proof}[Sketch of proof]
The original proof relies on Stone's theorem (and the spectral theorem, by proxy), i.e. the fact that $U_t$ has a resolution of the identity $\{E(\lambda)\}_{\lambda\in\R}$ (for a precise definition see Definition \ref{def:res-id} below) for which $U_t=\int_{\R}e^{it\lambda}\di E(\lambda)$. This leads to:
	\begin{align}
	(P^T-P)f
	&=
	\frac{1}{2T}\int_{-T}^TU_tf\,\di t-Pf
	=
	\frac{1}{2T}\int_{-T}^T\int_\R e^{it\lambda}\,\di E(\lambda)f\,\di t-Pf\notag\\
	&=
	\frac{1}{2T}\int_{-T}^T\int_{\R\setminus\{0\}} e^{it\lambda}\,\di E(\lambda)f\,\di t
	=
	\int_{\R\setminus\{0\}} \frac{\sin T\lambda}{T \lambda}\,\di E(\lambda)f.\label{eq:von-neumann}
	\end{align}
This last expression tends to $0$ as $T\to+\infty$.
\end{proof}

The strong convergence  $P^T\to P$ can be improved to uniform convergence if $H$ has a spectral gap:

\begin{theorem}[Ergodic theorem: case of spectral gap]\label{thm:gap}
Assume that there exists $\gamma>0$ such that $\sigma(H)\subset (-\infty,-\gamma]\cup\{0\}\cup[\gamma,+\infty)$. Then
	\begin{equation}\label{eq:spec-gap}
	\|P^T-P\|_{\Hcal\to\Hcal}\leq \gamma^{-1}T^{-1},\qquad\forall T>1.
	\end{equation}
\end{theorem}

\begin{proof}
See Remark \ref{rek:gap} below.
\end{proof}

\subsection{The spectral theorem}
Since the spectral theorem and the resolution of the identity of  self-adjoint operators play a central role in this paper, we recall some basic facts.

\begin{definition}[Resolution of the identity]\label{def:res-id}
Let  $H:D(H)\subset\Hcal\to\Hcal$ be a self-adjoint operator. Its associated \emph{resolution of the identity} $\{E(\lambda)\}_{\lambda\in\R}$  is a family of projection operators in $\Hcal$ with the property that, for each $\lambda\in\R$, the subspace $\Hcal^\lambda=E(\lambda)\Hcal$ is the largest closed subspace such that

{i.} $\Hcal^\lambda$ \emph{reduces} $H$, namely, $HE(\lambda)g=E(\lambda)Hg$ for every $g\in D(H)$. In particular, if $g\in D(H)$ then also $E(\lambda)g\in D(H)$.

{ii.}
$(Hu,u)_{\Hcal}\leq \lambda (u,u)_{\Hcal}$ for every $u\in\Hcal^\lambda\cap D(H)$.
\end{definition}
Now we are able to state the spectral theorem:
\begin{theorem}[Spectral theorem]
Let  $H:D(H)\subset\Hcal\to\Hcal$ be a self-adjoint operator and let  $\{E(\lambda)\}_{\lambda\in\R}$ be the associated resolution of the identity. Then $\{E(\lambda)\}_{\lambda\in\R}$ is unique, and the identity $H=\int_\R\lambda\,\di E(\lambda)$ holds.
\end{theorem}
In addition to the above, it is useful to state the definition of the spectral measure:
\begin{definition}[Spectral measure]
Given any $f,g\in \Hcal$ the resolution of the identity defines a complex function of bounded variation on the real line, given by
	\begin{equation*}\label{eq:spec-meas}
	\R\ni\lambda\mapsto(E(\lambda)f,g)_{\Hcal}.
	\end{equation*}
This  function gives rise to a complex measure (depending on $f,g$) called the \emph{spectral measure}.
\end{definition}


\begin{definition}[Density of states]
Let $\Xcal\subset\Hcal$ be some closed subspace. We call the bilinear form
\begin{equation*}
	\frac{\di}{\di\lambda}\left(E(\lambda)\cdot,\cdot\right)_{\Hcal}:\mathcal X \times \mathcal X\to\C
	\end{equation*}
the \emph{density of states of $H$ at $\lambda$ on the subspace $\Xcal$}.
\end{definition}

\subsection{Main results}
As mentioned above, we assume the opposite of a spectral gap: we assume that $\sigma(H)$ contains a neighborhood of $0$. However, we do not want to have ``too much'' spectrum near $0$. We make this precise as follows. Let $\{E(\lambda)\}_{\lambda\in\R}$ be the resolution of the identity of $H$. Our main assumption is:
\begin{assumption}\label{ass1}
There exist

\emph{i}. a Banach subspace $\Xcal\subset\Hcal$ which is dense in $\Hcal$  in the topology of $\Hcal$, is continuously embedded in $\Hcal$ (and therefore the norm $\|\cdot\|_\Xcal$ is stronger than the norm $\|\cdot\|_\Hcal$),

\emph{ii}. a real number $r\in(0,1)$ and a function $\psi:[-r,r]\to\R$ that is strictly positive a.e. on $I_r=[-r,r]$ such that the following bound of the DoS of $H$ holds:
	\begin{equation}\label{eq:dos-bound}
	\left|\frac{\di}{\di\lambda}\left(E(\lambda)f,g\right)_{\Hcal}\right|
	\leq
	\psi(\lambda)\|f\|_\Xcal\|g\|_\Xcal,\qquad\forall f,g\in\Xcal,\,\forall\lambda\in I_r\setminus\{0\},
	\end{equation}

\emph{iii}. a real number $q>0$ such that  $|\lambda|^{-q}\psi(\lambda)\in L^1(I_r)$.
\end{assumption}


Our main result, to be proved in Section \ref{sec:main}, is:
\begin{theorem}\label{thm:main}
Under Assumption A\ref{ass1}, letting $\ell = \min\{q,2\}$, the following uniform rate in von Neumann's ergodic theorem holds:
	\begin{equation}\label{eq:conv-rate}
	\|P^T-P\|_{\Xcal\to\Hcal}\leq \frac{C}{T^{\ell/2}},\qquad\forall T>1,
	\end{equation}
where $C$ is an explicit constant that does not depend on $T$.
\end{theorem}


When $\psi(\lambda)$ is a power of $\lambda$ we immediately have:

%
%

\begin{corollary}\label{cor:pol-log-bounds}
If $\psi$ has the form $\psi(\lambda) = c |\lambda|^{p-1}$ for some $c,p>0$, then \eqref{eq:conv-rate} holds with  $\ell = \min\{p-\eps,2\}$ for any $\eps>0$.
%
%
%
%
%
\end{corollary}

\begin{proof}[Proof of Corollary \ref{cor:pol-log-bounds}]
Considering Assumption \ref{ass1}, if $\psi(\lambda)=c|\lambda|^{p-1}$, in order for $|\lambda|^{-q}\psi(\lambda)$ to be integrable in a neighborhood of $0$, we need $q<p$. The assertion follows from applying Theorem \ref{thm:main}.
\end{proof}

\begin{remark}
We note that the best rate that one could expect is $T^{-1}$, as in the case of a spectral gap (see \cite[Remark 3]{Dzhulai2011} for a detailed proof). Hence, in Corollary \ref{cor:pol-log-bounds} the rate cannot be improved beyond $T^{-1}$ even  if $p>2$.
\end{remark}

\begin{remark}
As we show in the examples in Section \ref{sec:ex}, there is a delicate interplay between the choice of subspace $\Xcal\subset\Hcal$ and the bound $\psi(\lambda)$ one can obtain for the DoS, both appearing in \eqref{eq:dos-bound}:
\begin{itemize}\item In the case of the linear Schr\"odinger equation in $\R^d$, taking $\Xcal=L^{2,s}$ (a weighted $L^2$ space) leads to $\psi(\lambda)\sim\lambda^{-1/2}$ (cf. \eqref{eq:dens-states-1.1}), while taking $\Xcal=L^1\cap L^2$ leads to $\psi(\lambda)\sim|\lambda|^{\frac d2-1}$ (cf. \eqref{eq:dens-states-1.2})
\item In the case of the linear wave equation in $\R^d$, taking $\Xcal=L^{2,s}$  leads to $\psi(\lambda)\sim1$ (cf. \eqref{eq:dens-states-2.1}), while taking $\Xcal=L^1\cap L^2$ leads to $\psi(\lambda)\sim|\lambda|^{d-1}$ (cf. \eqref{eq:dens-states-2.2}).\end{itemize}
\end{remark}

\begin{remark}
In view of \cite{Dzhulai2011} (see discussion of previous results and Theorem \ref{thm:dk} below) these results are \emph{nearly} optimal, and might in fact be optimal. Establishing whether or not this is the case is the subject of ongoing research.
\end{remark}

\subsection{Previous results}
The idea that the spectrum encodes information about the dynamical system goes all the way back to von Neumann and his original proof of the mean ergodic theorem \cite{VonNeumann1932a}. Soon after, Riesz \cite{Riesz1938} provided an alternative proof which does not rely on the spectral theorem and is considered simpler. As a result, in most textbooks the spectral approach appears to have been lost.

Kachurovskii and coauthors have published extensively over the last 20 years on the topic of rates of convergence in ergodic theorems,  by revisiting von Neumann's original ideas, relying on the spectral theorem \cite{Kachurovskii1996,Kachurovskii1999,Kachurovskii2010a,Kachurovskii2010,Kachurovskii2011,Dzhulai2011}. We also mention \cite{Assani2007} where some of the techniques of \cite{Kachurovskii1996} were simplified. The survey \cite{Kachurovskii2016} provides a detailed overview of this sequence of results. Closest in spirit to our result is \cite{Dzhulai2011}, where  the authors  show:

\begin{theorem}[\cite{Dzhulai2011}]\label{thm:dk}
For any function $f\in\Hcal$ and $p\in[0,2)$, the two following conditions are equivalent:

{i}. $\left(\left(E(\lambda)-E(-\lambda)-E(\{0\})\right)f,f\right)_\Hcal\leq A\lambda^p$, for all $\lambda>0$,

{ii}. $\|(P^T-P)f\|_\Hcal\leq \frac{B}{T^{p/2}}\|f\|_\Hcal$, for all $T>0$,

\noindent where given $A>0$ one can compute $B=B(A,p)>0$ explicitly, and, conversely, given $B>0$ one can compute $A=A(B,p)>0$ explicitly (and these dependencies do not involve the function $f$).
\end{theorem}

Another important direction of research relates  to the appearance of  the \emph{Fej\'er kernel} $\sinc^2 T\lambda=\frac{\sin^2 T\lambda}{T^2\lambda^2}$ in the expression for the square of $P^T-P$ (see \eqref{eq:von-neumann}). This is addressed in \cite{Kachurovskii2018a,Kachurovskii2018,Kachurovskii2018c}, for example.

Finally, it is natural to compare our result to the well-known RAGE theorem (see e.g. \cite{Cycon1987}). It states that for any  compact operator $K$ and any $f\in\Hcal$,
\[
	\lim_{T \to +\infty} \frac{1}{T} \int_{0}^{T} \left\| K U_t P_{\mathrm{ac}} f \right\|^2 \di t = 0,
\]
where $P_{\mathrm{ac}}$ is the orthogonal projection onto the absolutely continuous subspace of $H$. Where Theorem \ref{thm:main} proves uniform convergence to the projection onto the kernel of the generator, the RAGE theorem proves a weak convergence to $0$ of the time average of the evolution of the continuous part of the spectrum.

\subsection*{Organization of the paper}
Section \ref{sec:main} is devoted to the proof of the main theorem, Theorem \ref{thm:main}. In Section \ref{sec:ex} we apply this to the linear Schr\"odinger and wave equations and obtain decay estimates for averages of solutions.

\section{Proof of the main result}\label{sec:main}

\begin{proof}[Proof of Theorem \ref{thm:main}]
Our starting point is the observation \cite[V-\S2.1]{Kato1995}  that if the bilinear form $\frac{\di}{\di\lambda}\left(E(\lambda)\cdot,\cdot\right)_{\Hcal}:\Xcal\times\Xcal\to\C$ is bounded at a given $\lambda\in\R$, then there exists a bounded operator $A(\lambda):\Xcal\to\Xcal^*$ satisfying
	\begin{equation*}
	\left<A(\lambda)f,g\right>=\frac{\di}{\di\lambda}\left(E(\lambda)f,g\right)_{\Hcal},\qquad \forall f,g\in\Xcal,
	\end{equation*}
where $\langle\cdot,\cdot\rangle$ is the $(\Xcal^*,\Xcal)$ dual space pairing. Moreover, the operator norm of $A(\lambda)$ shares the same bound as the bilinear form.
Now, recalling von Neumann's proof as sketched in \eqref{eq:von-neumann} above, we have
	\begin{equation*}
	(P^T-P)f
	=
	\int_{\R\setminus\{0\}} {\sinc (T\lambda)}\,\di E(\lambda)f
	\end{equation*}
where $\sinc\, x:=\sin x/x$. We split this integral as follows:
	\begin{equation*}
	\int_{\R\setminus\{0\}}{\sinc (T\lambda)}\,\di E(\lambda)f
	=
	\left(\int_{I_r\setminus\{0\}}+\int_{I_r^c}\right) {\sinc (T\lambda)}\,\di E(\lambda)f
	\end{equation*}
where $I_r=[-r,r]$. We start by estimating the second integral (``high frequency'' part), using the fact that ${\sinc^2 x}\leq |x|^{-2}$ and that projections onto different spectral parameters are mutually orthogonal:
	\begin{align*}
	\left\|\int_{I_r^c}{\sinc (T\lambda)}\,\di E(\lambda)f\right\|_{\Hcal}^2
	&=
	\int_{I_r^c}{\sinc^2 (T\lambda)}\,\di (E(\lambda)f,f)_{\Hcal}\\
	&\leq
	\frac{1}{T^2 r^2}\int_{\R}\,\di (E(\lambda)f,f)_{\Hcal}
	=
	\frac{1}{T^2 r^2}\|f\|_{\Hcal}^2
	\leq
	\frac{1}{T^2 r^2}\|f\|_{\Xcal}^2.
	\end{align*}
Now we turn to the first integral (``low frequency'' part). We use the estimate \eqref{eq:dos-bound} of the DoS, as well as the boundedness of $|x|^{\ell}{\sinc^2 x}$ for $\ell\in[0,2]$:
	\begin{align*}
	\left\|\int_{I_r\setminus\{0\}}{\sinc (T\lambda)}\,\di E(\lambda)f\right\|_{\Hcal}^2
	&=
	\int_{I_r\setminus\{0\}}{\sinc^2 (T\lambda)}\,\di (E(\lambda)f,f)_{\Hcal}\\
	&=
	\int_{I_r}{\sinc^2 (T\lambda)}\left<A(\lambda)f,f\right>\di\lambda\\
	&\leq
	\left(\int_{I_r} {\sinc^2 (T\lambda)}\psi(\lambda)\,\di\lambda\right)\|f\|_{\Xcal}^2\\
	&\leq
	\sup_{\lambda\in I_r}\left(|\lambda|^{\ell} {\sinc^2 (T\lambda)}\right)\left(\int_{I_r} |\lambda|^{-\ell}\psi(\lambda)\,\di\lambda\right)\|f\|_{\Xcal}^2.
	\end{align*}
Letting $\ell = \min\{q,2\}$, and using the fact that $|\lambda|^{-\ell}\leq|\lambda|^{-q}$ for $|\lambda|\leq r<1$, there holds
	\begin{align*}
	\left\|\int_{I_r\setminus\{0\}}{\sinc (T\lambda)}\,\di E(\lambda)f\right\|_{\Hcal}^2
	&\leq
	\sup_{\lambda\in I_r}\left(|\lambda|^{\ell}{\sinc^2 (T\lambda)}\right)\left(\int_{I_r} |\lambda|^{-q}\psi(\lambda)\,\di\lambda\right)\|f\|_{\Xcal}^2\\
	&=
	\frac{1}{T^{\ell}}\sup_{y\in I_{r T}}\left({|y|^{\ell}} {\sinc^2 y}\right)\Psi_{q}(r)\|f\|_{\Xcal}^2
	\leq
	\frac{1}{T^{\ell}}\Psi_{q}(r)\|f\|_{\Xcal}^2
	\end{align*}
where we have denoted  $\Psi_q(r):=\int_{I_r} |\lambda|^{-q}\psi(\lambda)\,\di\lambda$.
Altogether, both estimates lead to
	\begin{equation}\label{eq:rate-final}
	\|(P^T-P)f\|_{\Hcal}^2
	\leq
	\frac{1}{T^{\ell}}\left(\Psi_q(r)
	+
	\frac{1}{T^{2-{\ell}}r^2}\right)\|f\|_{\Xcal}^2
	\end{equation}
which completes the proof.
\end{proof}

\begin{remark}[The constant in \eqref{eq:rate-final}]
The constant in \eqref{eq:rate-final} is indeed uniformly bounded independent of $T$:
	\begin{equation*}
	C=\Psi_q(r)
	+
	\frac{1}{T^{2-{\ell}}r^2}
	\leq
	\Psi_q(r)
	+
	\frac{1}{r^2},\qquad\forall T>1.
	\end{equation*}
In fact, it even decreases slowly with $T$ and one could attempt to optimize it by letting $r$ tend to $0$ at an appropriate rate. However, there is nothing to be gained by doing this. This is simply an artefact due to the condition  $|\lambda|^{-q}\psi(\lambda)\in L^1(I_r)$ in Assumption \ref{ass1}: this is an open condition, in the sense that one could always increase $q$ slightly and this condition will still hold. Doing this will cause $C$ to converge to a constant independent of $T$.
\end{remark}

\begin{remark}[Spectral gap]\label{rek:gap}
In the case of a spectral gap \eqref{eq:spec-gap} immediately follows. Indeed, with gap of size $\gamma$ in the above proof $\psi$ is trivial and one has
	\begin{equation*}
	\|(P^T-P)f\|_{\Hcal}
	\leq
	\gamma^{-1}T^{-1}
	\|f\|_{\Hcal}.
	\end{equation*}
Note that in this case the subspace $\Xcal$ is no longer needed.
\end{remark}

\section{Examples}\label{sec:ex}
\subsection{The Laplace operator}
Let $\phi:[0,+\infty)\to[0,+\infty)$ be some continuous and strictly increasing function and define  $H=\phi(-\Delta)$ as a function of the Laplace operator  acting in $\Hcal=L^2(\R^d)$ with an appropriate domain for self-adjointness. Note that if $\phi(x)=x$ is the identity, then $-iH$ is the generator of the \emph{Schr\"odinger equation}:
	\begin{equation*}
	\left\{\begin{split}
	&\p_tf(t,x)=i\Delta f(t,x), &t\in\R,\,x\in\R^d,\\
	&f(0,x)=f_0(x),&x\in\R^d.
	\end{split}\right.
	\end{equation*}
Let $\{E(\lambda)\}_{\lambda\in\R}$ be the resolution of the identity of $H$. We use the fact that the Fourier transform is a unitary map relating $-\Delta$ to multiplication by $|\xi|^2$ in order to get:
	\begin{equation}\label{eq:ex1-bilin}
	\left(E(\lambda) f,g\right)_\Hcal
	=
	\int_{\phi(|\xi|^2)\leq\lambda}\widehat{f}(\xi)\overline{\widehat{g}(\xi)}\,\di\xi,\qquad\lambda\geq0.
	\end{equation}
Let us show how different choices of subspaces $\Xcal$  in the DoS estimate \eqref{eq:dos-bound} can give rise to different results.

\subsubsection{Hilbertian subspace}
Differentiating \eqref{eq:ex1-bilin} in $\lambda$ we get
	\begin{equation}\label{eq:ex1-dens}
	\frac{\di}{\di\lambda}\Big|_{\lambda=\lambda_0}\left(E(\lambda) f,g\right)_\Hcal
	=
	\int_{|\xi|=\sqrt{\phi^{-1}(\lambda_0)}}\widehat{f}(\xi)\overline{\widehat{g}(\xi)}\left|\nabla\left(\phi(|\xi|^2)\right)\right|^{-1}\,\di\sigma
	\end{equation}
where $\di\sigma$ is the Lebesgue (uniform) surface measure on the $d-1$-dimensional sphere of radius $\sqrt{\phi^{-1}(\lambda_0)}$.  The term $\left|\nabla\left(\phi(|\xi|^2)\right)\right|^{-1}
	=
	\frac{1}{2|\xi|\phi'(|\xi|^2)}$
comes from the coarea formula \cite[Appendix C3]{Evans2010}. An evaluation of the $L^2$ functions $\widehat{f}$ and $\widehat{g}$ on the hypersurface $\{|\xi|=\sqrt{\phi^{-1}(\lambda_0)}\}$ only makes sense if they  belong to any Sobolev space $H^s(\R^d)\subset L^2(\R^d)$ with $s>1/2$ by the trace lemma\footnote{This is not entirely optimal, since we are not making use of the fact that this hypersurface is in fact a sphere.}. The functions $\widehat{f}$ and $\widehat{g}$ belong to $H^s(\R^d)$ if and only if $f$ and $g$ belong to $L^{2,s}(\R^d)$, defined as
	\begin{equation*}
	L^{2,s}(\R^d):=\left\{f\in L^2(\R^d)\;:\;\|f\|_{L^{2,s}(\R^d)}^2:=\int_{\R^d}|f(x)|^2(1+|x|^2)^{s}\,\di x<+\infty\right\}.
	\end{equation*}
We therefore conclude that we can bound \eqref{eq:ex1-dens} using the  $L^{2,s}$-norms of $f$ and $g$, which are stronger than their $\Hcal$-norms:
	\begin{equation}\label{eq:dens-ex1}
	\left|\frac{\di}{\di\lambda}\Big|_{\lambda=\lambda_0}\left(E(\lambda) f,g\right)_\Hcal\right|\leq
	\frac{1}{2\sqrt{\phi^{-1}(\lambda_0)}\phi'(\phi^{-1}(\lambda_0))}
	\|{f}\|_{L^{2,s}}\|{g}\|_{L^{2,s}},
	\end{equation}
where  we denote $L^{2,s}$ rather than $L^{2,s}(\R^d)$ for brevity. Hence, using the notation of \eqref{eq:dos-bound}, we have that
	\begin{align*}
	\Xcal&=L^{2,s}(\R^d)\\
	\psi(\lambda)&=\frac{1}{2\sqrt{\phi^{-1}(\lambda)}\phi'(\phi^{-1}(\lambda))}.
	\end{align*}
In the case of the Schr\"odinger equation ($\phi=\id$) we get
	\begin{equation}\label{eq:dens-states-1.1}
	\psi(\lambda)=\frac{1}{2\sqrt{\lambda}}
	\end{equation}
and from Corollary \ref{cor:pol-log-bounds} we get a convergence rate of $T^{-\ell/2}$ for all $\ell < \frac{1}{2}$. Moreover, since $-\Delta$ has no eigenvalues in this setting (and, in particular, a trivial kernel), we conclude that
	\begin{equation*}\label{eq:schr-decay}
	\left\|\frac{1}{2T}\int_{-T}^Te^{-it\Delta}f_0\,\di t\right\|_{L^{2}}
	=
	\|P^Tf_0\|_{L^{2}}
	\leq
	CT^{-\ell/2}
	\|f_0\|_{L^{2,s}},\qquad\forall \ell<\frac12.
	\end{equation*}
This also implies that
	\begin{equation}\label{eq:schr-global}
	\|P^Tf_0\|_{L_T^q L_x^{2}([0,\infty)\times\R^d)}
	\leq
	C(q)
	\|f_0\|_{L_x^{2,s}(\R^d)},\qquad\forall q>4.
	\end{equation}
Restricting to any bounded domain $\Omega\subset\R^d$ we may take $L^2$ norms rather than weighted norms (the weight is uniformly bounded away from $0$ and $+\infty$ in $\Omega$) so we have
	\begin{equation*}\label{eq:schr-global-bounded}
	\|P^Tf_0\|_{L_T^qL_x^{2}([0,\infty)\times\Omega)}
	\leq
	C(q,\Omega)
	\|f_0\|_{L_x^{2}(\Omega)},\qquad\forall q>4.
	\end{equation*}

\subsubsection{Non-Hilbertian subspace}
Considering \eqref{eq:ex1-dens} again, we may change variables so that the integration takes place on the unit sphere in $\R^d$:
	\begin{align*}
	\frac{\di}{\di\lambda}\Big|_{\lambda=\lambda_0}\left(E(\lambda) f,g\right)_\Hcal
	&=
	\int_{|\xi|=\sqrt{\phi^{-1}(\lambda_0)}}\widehat{f}(\xi)\overline{\widehat{g}(\xi)}\left|\nabla\left(\phi(|\xi|^2)\right)\right|^{-1}\,\di\sigma\\
	&=
	\sqrt{\phi^{-1}(\lambda_0)}^{d-1}\int_{\Sbb^{d-1}}\frac{\widehat{f}\left(\tau\sqrt{\phi^{-1}(\lambda_0)}\right)\overline{\widehat{g}\left(\tau\sqrt{\phi^{-1}(\lambda_0)}\right)}}{2\sqrt{\phi^{-1}(\lambda_0)}\phi'(\phi^{-1}(\lambda_0))}\,\di\tau\\
	&=
	\frac{\sqrt{\phi^{-1}(\lambda_0)}^{d-2}}{2\phi'(\phi^{-1}(\lambda_0))}\int_{\Sbb^{d-1}}\widehat{f}\left(\tau\sqrt{\phi^{-1}(\lambda_0)}\right)\overline{\widehat{g}\left(\tau\sqrt{\phi^{-1}(\lambda_0)}\right)}\,\di\tau
	\end{align*}
where $\di\tau$ is the uniform measure on the unit sphere in $\R^d$.
Another way to make sense of the restriction of $L^2$ functions to a hypersurface is if they are bounded, i.e. one can bound:
	\begin{align}
	\left|\frac{\di}{\di\lambda}\Big|_{\lambda=\lambda_0}\left(E(\lambda) f,g\right)_\Hcal\right|
	&\leq
	\frac{|\Sbb^{d-1}|\sqrt{\phi^{-1}(\lambda_0)}^{d-2}}{2\phi'(\phi^{-1}(\lambda_0))}\|\widehat{f}\|_{L^\infty(\R^d)}\|\widehat{g}\|_{L^\infty(\R^d)}\notag\\
	&\leq
	\frac{|\Sbb^{d-1}|\sqrt{\phi^{-1}(\lambda_0)}^{d-2}}{2\phi'(\phi^{-1}(\lambda_0))}\|f\|_{L^1(\R^d)}\|g\|_{L^1(\R^d)}.\label{eq:dens-ex2}
	\end{align}
Thus we obtain
	\begin{align*}
	\Xcal&=L^1(\R^d)\cap L^2(\R^d)\quad\text{with norm}\quad\|\cdot\|_\Xcal:=\|\cdot\|_{L^1(\R^d)}+\|\cdot\|_{L^2(\R^d)}\\
	\psi(\lambda)&=\frac{|\Sbb^{d-1}|\sqrt{\phi^{-1}(\lambda)}^{d-2}}{2\phi'(\phi^{-1}(\lambda))}.
	\end{align*}
We again consider the Schr\"odinger case  ($\phi=\id$) where we obtain
	\begin{equation}\label{eq:dens-states-1.2}
	\psi(\lambda)=\frac12|\Sbb^{d-1}|\lambda^{\frac d2-1}.
	\end{equation}
From Corollary \ref{cor:pol-log-bounds} we get a convergence rate of $T^{-\ell/2}$ where $\ell= \min\{\frac{d}{2}-\eps,2\}$ for any $\eps >0$.
In particular, for dimensions $d\geq 5$ we obtain the optimal rate of convergence of  $T^{-1}$.
For any $f_0\in L^2 \cap L^1$, there holds
	\begin{equation*}\label{eq:schr-decay2}
	\left\|\frac{1}{2T}\int_{-T}^Te^{-it\Delta}f_0\,\di t\right\|_{L^2}
	=
	\|P^Tf_0\|_{L^2}
	\leq
	CT^{-\ell/2}
	\|f_0\|_{\Xcal},\quad \ell= \min\left\{\frac{d}{2}-\eps,2\right\}.
	\end{equation*}
This leads to the  global-in-time estimate for any $f_0\in L^2 \cap L^1$:
	\begin{equation}\label{eq:schr-global2}
	\|P^Tf_0\|_{L_T^qL^2_x([0,\infty)\times\R^d)}
	\leq
	C(q)
	\|f_0\|_{\Xcal},\qquad\forall q>\max\left\{\frac{4}{d},1\right\}.
	\end{equation}

\begin{remark}
It is natural to compare the  estimates \eqref{eq:schr-global} and \eqref{eq:schr-global2} with:

1) The well-known \emph{Strichartz estimates}
	\begin{equation*}
	\left\|e^{it\Delta/2}f_0\right\|_{L_t^qL_x^r(\R\times\R^d)}\leq C(d,q,r)\|f_0\|_{L_x^2(\R^d)}
	\end{equation*}
where $2\leq q,r\leq\infty$, $\frac 2q+\frac dr=\frac d2$ and $(q,r,d)\neq(2,\infty,2)$, see \cite{Tao2006a}.

2) \emph{Smoothing estimates}, such as
	\begin{equation*}
	\left\||D_x|^{1/2}e^{it\Delta}f_0\right\|_{L_t^2L^{2,-s}_{x}(\R\times\R^d)}
	\leq
	C(d)\|f_0\|_{L^2_x(\R^d)}
	\end{equation*}
where $s>1/2$, see \cite{Ben-Artzi1992}. A detailed comparison between these estimates is elusive at the present time, and is the subject of ongoing research.
\end{remark}

\subsection{The wave operator}
We now consider the linear, homogeneous wave equation
	\begin{equation*}
	\left\{\begin{split}
	&\p_{t}^2f(t,x)-\Delta f(t,x)=0, &t\in\R,\,x\in\R^d,\\
	&f(0,x)=f_0(x), \,\p_tf(0,x) = g_0(x) , &x\in\R^d.
	\end{split}\right.
	\end{equation*}
We let $\Hcal=L^2(\R^d)$ and consider the self-adjoint operator (with an appropriate domain) $H=-\Delta$.  We first need to convert the above problem into a first order system. We follow a well-known procedure: define
	\begin{equation}
	\label{def:fpm}
	f_\pm:=\frac12\left(\sqrt{H}f\pm i\p_tf\right).
	\end{equation}
Then we compute
	\begin{equation*}
	\p_tf_\pm
	=
	\frac12\left(\sqrt{H}\p_tf\pm i\p_t^2f\right)
	=
	\frac12\left(\sqrt{H}\p_tf\mp iHf\right)
	=
	\frac i2\sqrt{H}\left(-i\p_tf\mp \sqrt{H}f\right)
	=
	\mp i\sqrt{H}f_\pm.
	\end{equation*}
It follows that the vector
	\begin{equation*}
	\label{def:F}
	F(t,x):=\left(\begin{array}{c}f_+(t,x)\\f_-(t,x)
	\end{array}\right)
	\end{equation*}
satisfies the equation
	\begin{equation*}
	F'(t)
	=
	-iKF
\qquad\text{where}\qquad
	K=\left(\begin{array}{cc}\sqrt{H}&0\\0&-\sqrt{H}
	\end{array}
	\right).
	\end{equation*}
Denoting $\{E_{\sqrt{H}}(\lambda)\}_{\lambda\in\R}$ and $\{E_K(\lambda)\}_{\lambda\in\R}$ the resolutions of the identity of $\sqrt{H}$ and $K$, respectively, we first observe that $E_{-\sqrt{H}}(\lambda)=I-E_{\sqrt{H}}(-\lambda)$ so that
	\begin{equation*}
	E_K(\lambda)=E_{\sqrt{H}}(\lambda)\oplus(I-E_{\sqrt{H}}(-\lambda)),\qquad \forall\lambda\in\R.
	\end{equation*}
As in the case of the Schr\"odinger equation, we may consider two cases:
\subsubsection{Hilbertian subspace}
Inserting $\phi(H)=\sqrt{H}$ into the bound appearing in  \eqref{eq:dens-ex1} one finds that $\psi(\lambda)=1$ so that for $s>1/2$,
	\begin{equation}\label{eq:dens-states-2.1}
	\left|\frac{\di}{\di\lambda}\left(E_{\sqrt{H}}(\lambda) f,g\right)_{\Hcal}\right|\leq
	\|{f}\|_{L^{2,s}}\|{g}\|_{L^{2,s}},\qquad\forall\lambda\in\R.
	\end{equation}
This implies that
	\begin{equation*}
	\left|\frac{\di}{\di\lambda}\left(E_K(\lambda)F,G\right)_{\Hcal\oplus\Hcal}\right|\leq
	\left\|F\right\|_{L^{2,s}\oplus L^{2,s}}\left\|G\right\|_{L^{2,s}\oplus L^{2,s}},\qquad\forall\lambda\in\R.
	\end{equation*}
A bound on the DoS of the form $\psi(\lambda)=1$ leads to a convergence rate of $T^{-\ell/2}$ with $\ell <1$ from   Corollary \ref{cor:pol-log-bounds}. Therefore, noting that the kernel is empty:
	\begin{equation}\label{eq:wave-decay}
	\left\|\frac{1}{2T}\int_{-T}^Te^{-itK}F_0\,\di t\right\|_{\Hcal\oplus\Hcal}
	\leq
	CT^{-\ell/2}
	\|F_0\|_{L^{2,s}\oplus L^{2,s}},\qquad\forall \ell<1.
	\end{equation}
To obtain direct bounds for the average of the solution $f(t)$ of the wave equation, we use the identity $f = {H}^{-1/2}\left( f_{+} + f_{-}\right) $ to write
	\begin{align*}
	\frac{1}{2T} \int_{-T}^T f(t)\, \di t
	& =
	\frac{1}{2T} \int_{-T}^T {H}^{-1/2}\left( f_{+} + f_{-}\right)(t) \,\di t \\
	& =
	\frac{1}{2T} \int_{-T}^T {H}^{-1/2}\Big( \left( e^{-itK} F_0 \right)_1 + \left( e^{-itK} F_0 \right)_2 \Big) \,\di t.
	\end{align*}
Estimate \eqref{eq:wave-decay} leads then to
	\begin{align*}
	\left\| \frac{1}{2T} \int_{-T}^T f(t)\, \di t \right\|_{\Hcal}
	&\leq
	CT^{-\ell/2}
	\left(
	\left\|{H}^{-1/2} \left(F_0\right)_1\right\|_{L^{2,s}} + \left\|{H}^{-1/2} \left(F_0\right)_2\right\|_{L^{2,s}}
	\right)\notag\\
	&\leq
	CT^{-\ell/2}
	\left(
	\left\|f_0\right\|_{L^{2,s}} +\left\|{H}^{-1/2} g_0\right\|_{L^{2,s}}
	\right),\label{eq:wave-decay2}
	\end{align*}
with $\ell<1$ and where in the second inequality we have used \eqref{def:fpm}:
	\[
	2\left\|{H}^{-1/2}f_\pm(t=0)\right\|_{L^{2,s}}
	\leq
	\|f_0\|_{L^{2,s}} + \left\|{H}^{-1/2} g_0\right\|_{L^{2,s}}.
	\]
Denoting $P^{T}\left(f_0,g_0\right) := \frac{1}{2T} \int_{-T}^T f(t) \,\di t$
for brevity, we deduce  the global-in-time estimate
	\begin{equation}\label{eq:wave-decay-global}
		\left\| P^{T}\left(f_0,g_0\right) \right\|_{ L^{q}_{T}L^{2}_x([0,\infty)\times\R^d) }
		\leq
		C(q)
		\left(
		\|f_0\|_{L^{2,s}(\R^d)} + \left\|{H}^{-1/2} g_0\right\|_{L^{2,s}(\R^d)}
		\right),\qquad\forall q>2.
	\end{equation}
\subsubsection{Non-Hilbertian subspace}
As before, taking the functional space to be $\Xcal=L^1\cap L^2$ (with norm $\|\cdot\|_\Xcal=\|\cdot\|_{L^1}+\|\cdot\|_{L^2}$), \eqref{eq:dens-ex2}  with $\phi(H)=\sqrt H$ leads to
	\begin{equation*}
	\psi(\lambda)=\frac{|\Sbb^{d-1}|\sqrt{\phi^{-1}(\lambda)}^{d-2}}{2\phi'(\phi^{-1}(\lambda))}=|\Sbb^{d-1}||\lambda|^{d-1}
	\end{equation*}
which then gives the bound
	\begin{equation}\label{eq:dens-states-2.2}
	\left|\frac{\di}{\di\lambda}\left(E_{\sqrt{H}}(\lambda) f,g\right)_{\Hcal}\right|\leq
	|\Sbb^{d-1}||\lambda|^{d-1}\|{f}\|_{L^{1}}\|{g}\|_{L^{1}},\qquad\forall\lambda\in\R.
	\end{equation}
Consequently, we have
	\begin{equation*}
	\left|\frac{\di}{\di\lambda}\left(E_K(\lambda)F,G\right)_{\Hcal\oplus\Hcal}\right|\leq
	|\Sbb^{d-1}||\lambda|^{d-1}\left\|F\right\|_{L^{1}\oplus L^{1}}\left\|G\right\|_{L^{1}\oplus L^{1}},\qquad\forall\lambda\in\R.
	\end{equation*}
This bound leads to a convergence rate of $T^{-\ell/2}$ where $\ell =\min\{d-\eps,2\}$ for any $\eps>0$ from   Corollary \ref{cor:pol-log-bounds}. For $d\geq3$ we obtain the optimal rate. Therefore, noting that the kernel is empty:
	\begin{equation*}\label{eq:wave-decay3}
	\left\|\frac{1}{2T}\int_{-T}^Te^{-itK}F_0\,\di t\right\|_{\Hcal\oplus\Hcal}
	\leq
	CT^{-\ell/2}
	\|F_0\|_{\Xcal\oplus \Xcal},\qquad \ell=\min\{d-\eps,2\}.
	\end{equation*}
Following the steps from the previous subsection, this leads to the global-in-time estimate
	\begin{equation}\label{eq:wave-decay-global2}
		\left\| P^{T}\left(f_0,g_0\right) \right\|_{ L^{q}_{T}L^{2}_x([0,\infty)\times\R^d) }
		\leq
		C(q)
		\left(
		\|f_0\|_{\Xcal} + \left\|{H}^{-1/2} g_0\right\|_{\Xcal}
		\right),\qquad\forall q>\max\left\{\frac2d,1\right\}.
	\end{equation}

\begin{remark}
Here we compare our estimates to \emph{Strichartz estimates} for the wave equation (see \cite{Tataru2001}):
	\[
	\| f \|_{L^q_tL^p_x(\R\times\R^d)} \leq C(d,q,p,s) \left( \| f_0 \|_{H^{s}(\R^d)} + \| g_0 \|_{H^{s-1}(\R^d)} \right)
	\]

\noindent for triplets satisfying $2 \leq p,q \leq \infty$ and
	\[
	\frac{1}{q} + \frac{d}{p} = \frac{d}{2} - s, \qquad \frac{2}{q} + \frac{d-1}{p} \leq \frac{d-1}{2} .
	\]

\noindent In particular, we can compare \eqref{eq:wave-decay-global} with the Strichartz estimate for $(q,p,s) = (\infty,2,0)$:
	\[
	\| f \|_{L^{\infty}_tL^2_x(\R\times\R^d)} \leq C(d) \left( \| f_0 \|_{L^2_x(\R^d)} + \| g_0 \|_{H^{-1}_x(\R^d)} \right).
	\]
\end{remark}

\bibliography{library}

\begin{thebibliography}{10}

\bibitem{Assani2007}
I.~Assani and M.~Lin.
\newblock {On the one-sided ergodic Hilbert transform}.
\newblock {\em Contemp. Math.}, 430:21--39, 2007.

\bibitem{Ben-Artzi1992}
M.~Ben-Artzi and S.~Klainerman.
\newblock {Decay and regularity for the Schr{\"{o}}dinger equation}.
\newblock {\em J. d'Analyse Math{\'{e}}matique}, 58(1):25--37, dec 1992.

\bibitem{Cycon1987}
H.~L. Cycon, R.~G. Froese, W.~Kirsch, and B.~Simon.
\newblock {\em {Schr{\"{o}}dinger operators with application to quantum
  mechanics and global geometry}}.
\newblock Springer-Verlag, Berlin, 1987.

\bibitem{Dzhulai2011}
N.~A. Dzhulaii and A.~G. Kachurovskii.
\newblock {Constants in the estimates of the rate of convergence in von
  Neumann's ergodic theorem with continuous time}.
\newblock {\em Sib. Math. J.}, 52(5):824--835, oct 2011.

\bibitem{Evans2010}
L.~C. Evans.
\newblock {\em {Partial Differential Equations (Graduate Studies in
  Mathematics)}}.
\newblock American Mathematical Society, 2010.

\bibitem{Kachurovskii1996}
A.~G. Kachurovskii.
\newblock {The rate of convergence in ergodic theorems}.
\newblock {\em Russ. Math. Surv.}, 51(4):653--703, aug 1996.

\bibitem{Kachurovskii1999}
A.~G. Kachurovskii.
\newblock {On uniform convergence in the ergodic theorem}.
\newblock {\em J. Math. Sci.}, 95(5):2546--2551, jul 1999.

\bibitem{Kachurovskii2018}
A.~G. Kachurovskii.
\newblock {Fej{\'{e}}r integrals and the von Neumann ergodic theorem with
  continuous time}.
\newblock {\em Zap. Nauchn. Sem. S.-Peterburg. Otd. Mat. Inst. Steklov.},
  474:171--182, 2018.

\bibitem{Kachurovskii2018c}
A.~G. Kachurovskii and K.~I. Knizhov.
\newblock {Deviations of Fejer Sums and Rates of Convergence in the von Neumann
  Ergodic Theorem}.
\newblock {\em Dokl. Math.}, 97(3):211--214, 2018.

\bibitem{Kachurovskii2016}
A.~G. Kachurovskii and I.~V. Podvigin.
\newblock {Estimates of the rate of convergence in the von Neumann and Birkhoff
  ergodic theorems}.
\newblock {\em Trans. Moscow Math. Soc.}, 77:1--53, 2016.

\bibitem{Kachurovskii2018a}
A.~G. Kachurovskii and I.~V. Podvigin.
\newblock {Fej{\'{e}}r Sums for Periodic Measures and the von Neumann Ergodic
  Theorem}.
\newblock {\em Dokl. Math.}, 98(1):344--347, 2018.

\bibitem{Kachurovskii2010a}
A.~G. Kachurovskii and A.~V. Reshetenko.
\newblock {On the rate of convergence in von Neumann's ergodic theorem with
  continuous time}.
\newblock {\em Sb. Math.}, 201(4):493--500, 2010.

\bibitem{Kachurovskii2010}
A.~G. Kachurovskii and V.~V. Sedalishchev.
\newblock {On the constants in the estimates of the rate of convergence in von
  Neumann's ergodic theorem}.
\newblock {\em Math. Notes}, 87(5-6):720--727, jul 2010.

\bibitem{Kachurovskii2011}
A.~G. Kachurovskii and V.~V. Sedalishchev.
\newblock {Constants in estimates for the rates of convergence in von Neumann's
  and Birkhoff's ergodic theorems}.
\newblock {\em Sb. Math.}, 202(8):1105--1125, 2011.

\bibitem{Kato1995}
T.~Kato.
\newblock {\em {Perturbation Theory for Linear Operators}}.
\newblock Springer-Verlag, 1995.

\bibitem{Riesz1938}
F.~Riesz.
\newblock {Some Mean Ergodic Theorems}.
\newblock {\em J. London Math. Soc.}, s1-13(4):274--278, oct 1938.

\bibitem{Tao2006a}
T.~Tao.
\newblock {\em {Nonlinear Dispersive Equations: Local and Global Analysis}}.
\newblock 2006.

\bibitem{Tataru2001}
D.~Tataru.
\newblock {Strichartz estimates for second order hyperbolic operators with
  nonsmooth coefficients, II}.
\newblock {\em Am. J. Math.}, 123(3):385--423, 2001.

\bibitem{VonNeumann1932a}
J.~von Neumann.
\newblock {Proof of the quasi-ergodic hypothesis}.
\newblock {\em Proc. Natl. Acad. Sci.}, 18(2):70--82, 1932.

\end{thebibliography}
\bibliographystyle{abbrv}
\end{document}